\documentclass[12pt]{amsart}

\newtheorem{theorem}{Theorem}

\newtheorem{remark}{Remark}

\numberwithin{equation}{section}

\usepackage{bm}

\begin{document}

\title[remarks on positive solutions]{Remarks on positive solutions to nonlinear problems and numerical methods}
\author{Y.~Adachi, Novrianti$^{\star}$ and O.~Sawada}
\address{Applied Physics Course, Faculty of Engineering, Gifu University, Yanagido 1-1, Gifu, 501-1193, Japan}
\email{${}^\star$ x3912006@edu.gifu-u.ac.jp \quad (corresponding author)}
%\subjclass[2010]{35K57, 35B50}

\baselineskip=16pt

%%%%%%%%%%%%%%%%%%%%%%%%%%%%%%%%%%%%%%%%%%%%%%%%%%%%%%%%%%
%
% abstract
%
%%%%%%%%%%%%%%%%%%%%%%%%%%%%%%%%%%%%%%%%%%%%%%%%%%%%%%%%%%

\begin{abstract}
The existence of positive solutions to the system of ordinary differential equations related to the Belousov-Zhabotinsky reaction is established. The key idea is to use successive approximation of solutions, ensuring its positivity. To obtain the positivity and invariant region for numerical solutions, the system is discretized as difference equations of explicit form, employing operator splitting methods with linear stability conditions.
\end{abstract}

\maketitle

%%%%%%%%%%%%%%%%%%%%%%%%%%%%%%%%%%%%%%%%%%%%%%%%%%%%%%%%%%
%
% introduction
%
%%%%%%%%%%%%%%%%%%%%%%%%%%%%%%%%%%%%%%%%%%%%%%%%%%%%%%%%%%

\section{Introduction}
\noindent We consider the reaction diffusion equations of the Keener-Tyson model for the Belousov-Zhabotinsky reaction in the whole space ${\mathbb R}^n$, $n \in {\mathbb N}$:
\[
  \left\{ \begin{array}{l}
    \partial_t u = \Delta u + u (1-u)/\varepsilon - h v (u-q)/(u+q), \\
    \partial_t v = d \Delta v - v + u.
  \end{array} \right. \leqno{\rm{(BZ)}}
\]
See \cite{KT86}. Here, $u = u(x, t)$ and $v = v(x, t)$ stand for the unknown scalar functions at $x \in {\mathbb R}^n$ and $t > 0$ which denote the concentrations in a vessel of $\rm{HBrO_2}$ and $\rm{Ce^{4+}}$, respectively. In \cite{Chen00}, an example of constants is listed-up as $\varepsilon \approx 0.032$, $q \approx 2.0 \times 10^{-4}$ and $d \approx 0.6 \times \varepsilon$; $h := \rho/\varepsilon$ and $\rho \approx 1/2$ stands for the excitability which governs dynamics of a pattern formulation. It has been used the notation of differentiation; $\partial_t := \partial/\partial t$, $\Delta := \sum_{i=1}^n \partial_i^2$, $\partial_i := \partial/\partial x_i$ for $i = 1, \ldots, n$.

In \cite{KNST19}, the time-global existence of positive unique smooth solutions to the Cauchy problem of (BZ) was established. They also obtained that $S := (q, \bar u)^2$ is an invariant region, where $\bar u \in (q, 1)$ is a root of $u (1 - u) (u + q) - \varepsilon h q (u - q) = 0$. That means, if the initial datum $\big( u_0, v_0 \big) \in S$, then the solution $\big( u, v \big) \in S$ for $t > 0$. The aim of this paper is to establish the similar results for numerical solutions to the difference equations of discretized (BZ).

The key idea in \cite{KNST19} is to use the following successive approximation:
\begin{align*}
  \partial_t u_{\ell+1} & = \Delta u_{\ell+1} + u_\ell (1-u_{\ell+1})/\varepsilon - h v_\ell (u_{\ell+1}-q)/(u_\ell+q), \\
  \partial_t v_{\ell+1} & = d \Delta v_{\ell+1} - v_{\ell+1} + u_\ell
\end{align*}
for $\ell \in {\mathbb N}$ with $u_{\ell+1}|_{t=0} = u_0$ and $v_{\ell+1}|_{t=0} = v_0$, starting at $u_1 := e^{t \Delta} u_0$ and $v_1 := e^{-t} e^{t \Delta} v_0$ with non-negative initial datum $u_0, v_0 \in BUC({\mathbb R}^n)$. The virtue of this scheme is to be ensured that $u_\ell \geq 0$ and $v_\ell \geq 0$ for all $\ell \in {\mathbb N}$, automatically. So, one can obtain the time-local non-negative classical solutions to (BZ) as $u = \lim_{\ell \to \infty} u_\ell$, $v = \lim_{\ell \to \infty} v_\ell$. We would emphasize that this technique can be applicable to construct positive (or, non-negative) solutions to ordinary differential equations and positive numerical solutions to finite difference equations.

It has been known that Mimura and his collaborators obtained the positive numerical solutions to some reaction diffusion equations; see \cite{MKY71, MN72}. In particular, our scheme is similar to that in \cite{MN72}, even though they did not employ operator splitting methods. Moreover, there are many literatures on structure-preserving numerical methods for partial differential equations; see e.g. \cite{FM11} and references therein. However, it seems to be new that the scheme leads us an invariant region for numerical solutions. This fact essentially follows from the maximum principle for numerical solutions to discretized reaction diffusion equations.

This paper is organized as follows. In Section~2, we shall discuss the existence of positive solutions to the system of ordinary differential equations. Section~3 will be devoted to give a difference equations for positive solutions. We state results on the discretization of (BZ) for numerical solutions and an invariant region in Section~4.

%%%%%%%%%%%%%%%%%%%%%%%%%%%%%%%%%%%%%%%%%%%%%%%%%%%%%%%%%%
%
% Acknowledgment
%
%%%%%%%%%%%%%%%%%%%%%%%%%%%%%%%%%%%%%%%%%%%%%%%%%%%%%%%%%%

\vspace{2mm}
\noindent {\bf Acknowledgment.} The authors would like to express their gratitude to Professor Hiroyuki Usami, Professor Shinya Miyajima and Professor Shintaro Kondo for letting them know several techniques and results on ordinary differential equations and numerical methods. The authors would also like to express their gratitude to Professor Yoshihisa Morita for letting them know the results by Mimura and his collaborators \cite{MKY71, MN72}.

%%%%%%%%%%%%%%%%%%%%%%%%%%%%%%%%%%%%%%%%%%%%%%%%%%%%%%%%%%
%
% Section 2.
%
%%%%%%%%%%%%%%%%%%%%%%%%%%%%%%%%%%%%%%%%%%%%%%%%%%%%%%%%%%

\section{Ordinary differential equations}
\noindent In this section, the construction of time-local positive solutions to the system of first order ordinary differential equations (ODE) is discussed. Let $m \in {\mathbb N}$, we deal with the following system of nonlinear ODE:
\[
  u_i' = - f_i (\bm{u}) u_i + g_i (\bm{u}), \quad t > 0, \quad u_i(0) = a_i, \quad i = 1, \ldots, m. \leqno{({\rm{P}})}
\]
Throughout this paper, for simplicity of notation, $t = 0$ is the initial time, ${}' := d/dt$, $\bm{u} := \big( u_1, \ldots, u_m \big)$. Here, $u_i = u_i (t)$ are unknown functions for $t > 0$ and $i = 1, \ldots, m$. Besides, $a_i \geq 0$ are given initial data, $f_i \geq 0$ and $g_i \geq 0$ are given function. We often rewrite (P) into the following vector valued ODE:
\[
  \bm{u}' = - F (\bm{u}) \bm{u} + \bm{g} (\bm{u}), \quad t > 0, \quad \bm{u} (0) = \bm{a}. \leqno{({\rm{P'}})}
\]
Here, we have denoted $\bm{a} := \big( a_1, \ldots, a_m \big)$, $\bm{g} := \big( g_1, \ldots, g_m \big)$, and  $F$ is the diagonal $m \times m$ matrix whose $(i, i)$-component is $f_i$.

When $m = 2$, $u := u_1$, $v := u_2$ and
\[
  F := \left( \begin{array}{cc}
    u/\varepsilon + hv/(u+q) & 0 \\
    0 & 1
  \end{array} \right), \quad
  \bm{g} := \left( \begin{array}{c}
    u/\varepsilon + hqv/(u+q) \\
    u
  \end{array} \right)
\]
are taken, then (P) is equivalent to the uniform-in-space (BZ). The model problem (P) is often used to describe the dynamics of nonlinear chemical or biological systems, for example, the Lotka-Volterra type equations of predator-prey models with density-dependent inhibition (Holling's type II), and the Gierer-Meinhardt model. We especially treat fractional nonlinear terms, and the denominator takes on the value of zero for negative solutions. Hence, the positivity of solutions to (P) is strongly required, and so is even in its approximation.

We shall state the main results in this paper.

%%%%%%%%%%%%%%%%%%%%%%%%%%%%%%%%%%%%%%%%%%%%%%%%%%%%%%%%%%
%
% Theorem 1. ODE
%
%%%%%%%%%%%%%%%%%%%%%%%%%%%%%%%%%%%%%%%%%%%%%%%%%%%%%%%%%%

\begin{theorem}\label{ode}
If $f_i, g_i \geq 0$ are local Lipschitz continuous and $a_i \geq 0$, then there exists a time-local unique solution $u_i \geq 0$ to {\rm{(P)}}.
\end{theorem}

\begin{proof}
Let $a_i \geq 0$ for $i = 1, \ldots, m$. For the sake of simplicity, let us assume that $\bm{a} \neq \bm{0}$ and $g_i (\bm{v}) > 0$ for $\bm{v} \neq \bm{0}$. Making the approximation sequences $\big\{ u_i^\ell \big\}_{\ell = 1}^\infty$ for $i = 1, \ldots, m$, we begin with $\bm{u}^1 (t) := \bm{a}$ for $t \geq 0$. For each $\ell \in {\mathbb N}$, we successively define $\bm{u}^{\ell + 1}$ as the solution to the system of linear non-autonomous ODE:
\[
    \big( \bm{u}^{\ell+1} \big)' = -F (\bm{u}^\ell) \bm{u}^{\ell+1} + \bm{g} (\bm{u}^\ell), \quad t > 0, \quad \bm{u}^{\ell+1}(0) =\bm{a} \leqno{\rm{(SA)}}
\]
with vectors of non-negative $\bm{u}^\ell$ and $\bm{a}$. Note that (SA) is equivalent to the following integral equation
\[
  \bm{u}^{\ell+1} (t) = \bm{a} - \int_0^t F ( \bm{u}^\ell (s) ) \bm{u}^{\ell+1} (s) ds + \int_0^t \bm{g} (\bm{u}^\ell (s)) ds. \leqno{\rm{(INT)}}
\]
Heuristically, if $F$ is a constant matrix, then $\bm{v}' = - F \bm{v}$, $t > 0$, $\bm{v}(0) = \bm{a}$ admits a solution $\bm{v} (t) = e^{- F t} \bm{a}$. In this situation, we thus have
\[
  \bm{u}^{\ell+1} (t) = e^{- F t} \bm{a} + \int_0^t e^{- F (t-s)} \bm{g} (\bm{u}^\ell (s)) ds.
\]
For general matrix-valued functions $F$, one may construct $\bm{u}^{\ell+1}$ for each $\ell \in {\mathbb N}$ by perturbation theory, at least time-locally.

Obviously, $u_i^1 \geq 0$ for $i = 1, \ldots, m$ and $\| \bm{u}^1 (t) \| = \| \bm{a} \|$ for $t \geq 0$. Here, we have used the max norm for vectors $\| \bm{v} \| := \max_{i = 1, \ldots, m} | v_i |$ for $\bm{v} := \big( v_1, \ldots, v_m \big)$, as well as to matrices $\| F \| := \max_{i, j = 1, \ldots, m} |f_{ij}|$ for $F := \big( f_{ij} \big)$. In what follows, we will show the positivity and boundedness of $u_i^{\ell + 1}$ by induction in $\ell$. For $\bm{u}^2$, it holds true that
\begin{align*}
  \| \bm{u}^2 (t) \| & \leq \| \bm{a} \| + \int_0^t \| F (\bm{u}^1 (s)) \| \cdot \| \bm{u}^2 (s) \| ds + \int_0^t \| \bm{g} (\bm{u}^1 (s)) \| ds \\
  & \leq \| \bm{a} \| + \| F (\bm{a}) \| \cdot t \cdot \max_{0 \leq s \leq t} \| \bm{u}^2 (s) \| + t \cdot \| \bm{g} (\bm{a}) \|.
\end{align*}
Taking $\max_{0 \leq t \leq \tau}$ in both hand side, we have
\[
  \| \bm{u}^2 (t) \| \leq 2 \| \bm{a} \| \quad \text{for} \quad t \in [0, T_2]
\]
with $T_2 := \min \big\{ 1/(3 \| F (\bm{a}) \|), \| \bm{a} \|/(3 \| \bm{g} (\bm{a}) \|) \big\}$. In addition, we can also obtain that $u_i^2 \geq 0$. Indeed, let us assume that there exists a $t_\ast \in (0, T_2]$ such that $u_i^2 (t_\ast) = 0$ for some $i = 1, \ldots, m$. Without loss of generality, $t_\ast$ is the first time when $u_i^2$ touches $0$. So, at $t_\ast$, we see that $(u_i^2)' \leq 0$, $f_i (\bm{u}^1) {u}_i^2 = 0$, and $g_i (\bm{u}^1) > 0$. This contradicts to the fact that $\bm{u}^2$ is a solution to (SA) with $\ell = 1$.

Let $\ell \geq 2$. Assume that $\| \bm{u}^\ell (t) \| \leq 2 \| \bm{a} \|$ and $u_i^\ell (t) \geq 0$ hold for $t \in [0, T_0]$ and $i = 1, \ldots, m$, where $T_0 > 0$ will be determined later. We now argue on $\bm{u}^{\ell+1}$. By assumption, it is easy to see that
\begin{align*}
  \| \bm{u}^{\ell+1} (t) \| & \leq \| \bm{a} \| + \int_0^t \| F (\bm{u}^\ell (s)) \| \cdot \| \bm{u}^{\ell+1} (s) \| ds + \int_0^t \| \bm{g} (\bm{u}^\ell (s)) \| ds \\
  & \leq \| \bm{a} \| + M_f \cdot t \cdot \max_{0 \leq s \leq t} \| \bm{u}^{\ell+1} (s) \| + t \cdot M_g \\
  & \leq 2 \| \bm{a} \| \quad \text{for} \quad t \in [0, T_0]
\end{align*}
with $T_0 := \min \big\{ 1/(3 M_f), \| \bm{a} \|/(3 M_g) \big\}$, where
\[
  M_f := \sup_{\| \bm{v} \| \leq 2 \| \bm{a} \|} \| F (\bm{v}) \|, \quad M_g := \sup_{\| \bm{v} \| \leq 2 \| \bm{a} \|} \| \bm{g} (\bm{v}) \|.
\]
In addition, we can also see that $u_i^{\ell+1} \geq 0$ for $i$ by the same contradiction argument above. This means that $\| \bm{u}^\ell (t) \| \leq 2 \| \bm{a} \|$ and $u_i^\ell \geq 0$ hold for all $\ell \in {\mathbb N}$ and $i = 1, \ldots, m$ for $t \in [0, T_0]$.

It is straightforward to get the continuity of solutions. One may also see that $\big\{ \bm{u}_\ell \big\}_{\ell = 1}^\infty$ is a Cauchy sequence in $C([0, T_0]; {\mathbb R}^m)$. So, the limit $\big( u_1 (t), \ldots, u_m (t) \big) = \bm{u} (t) = \lim_{\ell \to \infty} \bm{u}^\ell (t)$ exists for $t \in [0, T_0]$, and satisfies (P). Note that $u_i (t) \geq 0$ for $i = 1, \ldots, m$ by construction. The uniqueness follows from Gronwall's inequality, directly.
\end{proof}

In Theorem~\ref{ode}, it is not needed to use neither the existence of stable solutions to (P), comparison principle, nor a priori estimates by Lyapunov functions.

%%%%%%%%%%%%%%%%%%%%%%%%%%%%%%%%%%%%%%%%%%%%%%%%%%%%%%%%%%
%
% Section 3. Difference equations
%
%%%%%%%%%%%%%%%%%%%%%%%%%%%%%%%%%%%%%%%%%%%%%%%%%%%%%%%%%%

\section{Difference equations}
\noindent We will argue the numerical algorithm for positive solutions. We first discuss a discretization of (P). To obtain positive solutions, our proposal is to choose the following difference equations, mixing the forward and backward Euler methods:
\[
  \frac{\bm{u}^{k+1} - \bm{u}^k}{{\it\Delta} t} = - F (\bm{u}^k) \bm{u}^{k+1} + \bm{g} (\bm{u}^k), \quad k \in {\mathbb N}_0 := {\mathbb N} \cup \{ 0 \}, \leqno{({\rm{DE}})}
\]
where $\bm{u}^k = \big( u_1^k, \ldots, u_m^k \big)$, $t_k := k {\it\Delta} t$ for ${\it\Delta} t > 0$ and $u_i^0 = a_i \geq 0$ for $i = 1, \ldots, m$. Clearly, (DE) is a mimic of (SA). In addition, we obviously see that the numerical solution $\bm{u}^k$ to (DE) tends to the solution $\bm{u} (t)$ to (P) at $t = t_k$ for each $k$ as ${\it\Delta} t \to 0$.

%%%%%%%%%%%%%%%%%%%%%%%%%%%%%%%%%%%%%%%%%%%%%%%%%%%%%%%%%%
%
% Theorem 2. DE
%
%%%%%%%%%%%%%%%%%%%%%%%%%%%%%%%%%%%%%%%%%%%%%%%%%%%%%%%%%%

\begin{theorem}\label{de}
If $f_i, g_i \geq 0$ are Lipschitz continuous, ${\it\Delta} t > 0$ and $u_i^0 \geq 0$, then the numerical solution $u_i^k \geq 0$ to {\rm{(DE)}} exists for $k \in {\mathbb N}$.
\end{theorem}

\begin{proof}[Proof of Theorem~\ref{de}]
We rewrite (DE) into the explicit form as
\[
  u_i^{k+1} = \frac{u_i^k + g_i (\bm{u}^k) {\it\Delta} t}{1 + f_i (\bm{u}^k) {\it\Delta} t}, \quad k \in {\mathbb N}_0, \quad i = 1, \ldots, m.
\]
So, $u_i^{k+1} \geq 0$, if $u_i^k \geq 0$. Thus, one can prove it by induction.
\end{proof}

The advantage of Theorem~\ref{de} is to take arbitrary large ${\it\Delta} t$. 

The spirit of (DE) is still valid on the numerical methods for construction of positive solutions to the partial differential equations (PDE) of parabolic type. For simplicity, let $n = m = 1$. We consider the discretization $u_j^k$ of $u (x_j, t_k)$ for $x_j := j {\it\Delta} x$ and $t_k := k {\it\Delta} t$ which satisfies
\[%begin{equation}\label{at-once}
  \frac{u_j^{k+1} - u_j^k}{{\it\Delta} t} = d \frac{u_{j+1}^k - 2 u_j^k + u_{j-1}^k}{{\it\Delta} x^2} - f (u_j^k) u_j^{k+1} + g (u_j^k)
\]%end{equation}
with initial data $u_j^0 \geq 0$ for all $j$. So, it is easy to see that $u_j^k \geq 0$ for all $j$ and $k$, provided if the linear stability condition ${\it\Delta} t/{\it\Delta} x^2 \leq 1/(2 d)$ for $d > 0$ in the Lax-Richtmyer sense is assumed. Note that the similar scheme has also been introduced by Mimura in \cite{MKY71, MN72} for ensuring the postivity of numerical solutions, basically.

We sometimes employ the algorithm of operator splitting methods (OSM) for solving the discretized reaction diffusion equation. For using OSM, the linear stability conditions are also required; see Section~4.

%%%%%%%%%%%%%%%%%%%%%%%%%%%%%%%%%%%%%%%%%%%%%%%%%%%%%%%%%%
%
% Section 4. Numerical solutions to (BZ)
%
%%%%%%%%%%%%%%%%%%%%%%%%%%%%%%%%%%%%%%%%%%%%%%%%%%%%%%%%%%

\section{Numerical solutions to (BZ)}
\noindent We shall establish an invariant region for numerical solutions to (BZ). For the sake of simplicity, let $n = 1$, and let us concern (BZ) in bounded interval $x \in [0, L]$ for $L > 0$ with the homogeneous Neumann boundary conditions $\partial_x u (0, t) = \partial_x u (L, t) = 0$ or, the periodic boundary conditions $u (x, t) = u (x+L, t)$ for $t > 0$. For discretization of (BZ), we put $u^k_j \approx u (x_j, t_k)$ and $v^k_j \approx v (x_j, t_k)$ for $j = 0, \ldots, J$ and $k \in {\mathbb N}_0$, taking the average of integration. Here, $J \in {\mathbb N}$, $x_j := j {\it\Delta} x$, $t_k := k {\it\Delta} t$ for ${\it\Delta} x > 0$ and ${\it\Delta} t > 0$; $L = J \it\Delta x$.

In here, we select the operator splitting methods. For the discretization of uniform-in-space (BZ) with (DE) algorithm, let us consider
\[
  \left\{ \begin{array}{l}
    \displaystyle \frac{u_j^{k+1}-u_j^k}{{\it\Delta} t} = \frac{u_j^k (1-u_j^{k+1})}{\varepsilon} - h v_j^k \frac{u_j^{k+1}-q}{u_j^k+q}, \\[10pt]
    \displaystyle \frac{v_j^{k+1}-v_j^k}{{\it\Delta} t} = -v_j^{k+1} + u_j^k
  \end{array} \right. \leqno{\rm{(D_o)}}
\]
for $j = 1, \ldots, J-1$ and $k \in {\mathbb N}_0$. On the other hand, for the discretization of the heat equations, we use the standard FTCS (forward difference for time and second-order central difference for space)
\[
  \left\{ \begin{array}{l}
    \displaystyle \frac{\tilde u_j^{k+1} - \tilde u_j^k}{{\it\Delta} t} = \frac{\tilde u_{j+1}^k - 2 \tilde u_j^k + \tilde u_{j-1}^k}{{\it\Delta} x^2}, \\[10pt]
    \displaystyle \frac{\tilde v_j^{k+1} - \tilde v_j^k}{{\it\Delta} t} = d \frac{\tilde v_{j+1}^k - 2 \tilde v_j^k + \tilde v_{j-1}^k}{{\it\Delta} x^2}
  \end{array} \right. \leqno{\rm{(D_h)}}
\]
for $j = 1, \ldots, J-1$ and $t \in {\mathbb N}_0$; at $j = 0$ and $j = J$, we give certain definition by boundary conditions. Our algorithm is to solve alternate $\rm{(D_o)}$ and $\rm{(D_h)}$. That is to say, a pair of the series $\big\{ u_j^k, v_j^k \big\}$ is given as
\begin{enumerate}
\item Put $u_j^0 \approx u_0 (x_j)$ and $v_j^0 \approx v_0 (x_j)$, the average of integration.
\item Construct $u_j^1, v_j^1$ by $\rm{(D_o)}$ with $k = 0$.
\item Construct $\tilde u_j^1, \tilde v_j^1$ by $\rm{(D_h)}$ with $\tilde u_j^0 := u_j^1$ and $\tilde v_j^0 := v_j^1$.
\item Construct $u_j^2, v_j^2$ by $\rm{(D_o)}$ with $u_j^1 := \tilde u_j^1$ and $v_j^1 := \tilde v_j^1$.
\item Construct $\tilde u_j^2, \tilde v_j^2$ by $\rm{(D_h)}$ with $\tilde u_j^1 := u_j^2$ and $\tilde v_j^1 := v_j^2$.
\item Repeat this process.
\end{enumerate}

If $d = 0$, then we skip the steps of construction $\tilde v_j^k$, that is, $\tilde v_j^k := v_j^k$. We will state the results on numerical solutions to discretized (BZ).

%%%%%%%%%%%%%%%%%%%%%%%%%%%%%%%%%%%%%%%%%%%%%%%%%%%%%%%%%%
%
% Theorem 3. invariant region
%
%%%%%%%%%%%%%%%%%%%%%%%%%%%%%%%%%%%%%%%%%%%%%%%%%%%%%%%%%%

\begin{theorem}\label{ir}
Let $\varepsilon, h, {\it\Delta} t, {\it\Delta} x > 0$, $d \geq 0$ and $q \in (0, 1)$. Define $u_j^k, v_j^k$ as numerical solutions to alternate $\rm{(D_o)}$ and $\rm{(D_h)}$. If $u_j^0, v_j^0 \in (q, 1)$ for $j$, then $u_j^k, v_j^0 \in (q, 1)$ for $j$ and $k$, provided if ${\it\Delta} t/{\it\Delta} x^2 \leq 1/\max \{ 2, 2 d \}$.
\end{theorem}

\begin{proof}
By Theorem~\ref{de} and the linear stability conditions, it holds that $u_j^k, v_j^k \geq 0$ for all $j$ and $k$. The induction in $k$ is used. Let $u_j^k, v_j^k \in (q, 1)$. We first check that $u_j^{k+1}, v_j^{k+1} > q$ by $\rm{(D_o)}$. It turns out that
\begin{align*}
  u_j^{k+1}-q & = \frac{u_j^k + u_j^k {\it\Delta} t/\varepsilon + h q v_j^k {\it\Delta} t/(u_j^k+q)}{1 + u_j^k {\it\Delta} t/\varepsilon + h v_j^k {\it\Delta} t/(u_j^k+q)} - q \\
  & = \frac{(u_j^k - q) + (1 - q) u_j^k {\it\Delta} t/\varepsilon}{1 + u_j^k {\it\Delta} t/\varepsilon + h v_j^k {\it\Delta} t/(u_j^k+q)} > 0
\end{align*}
by $u_j^k > q$ and $q \in (0, 1)$. Similarly, we have
\[
  v_j^{k+1}-q = \frac{v_j^k + u_j^k {\it\Delta} t}{1 + {\it\Delta} t} - q
  = \frac{(v_j^k - q) + (u_j^k - q) {\it\Delta} t}{1 + {\it\Delta} t} > 0.
\]
One can also easily see that
\begin{align*}
  1-u_j^{k+1} & = \frac{(1-u_j^k) + (1 - q) h v_j^k {\it\Delta} t/(u_j^k+q)}{1 + u_j^k {\it\Delta} t/\varepsilon + h v_j^k {\it\Delta} t/(u_j^k+q)} > 0, \\
  1-v_j^{k+1} & = \frac{(1-v_j^k)+(1-u_j^k){\it\Delta} t}{1+{\it\Delta} t} > 0.
\end{align*}

On $\rm{(D_h)}$, it is well-known that the linear stability condition produces the maximum principle for numerical solution.
This completes the proof.
\end{proof}

\begin{remark}{\rm
{\rm{(i)}}~This assertion implies that $S_{\it\Delta} := (q, 1)^2$ is an invariant region for numerical solutions to alternate $\rm{(D_o)}$ and $\rm{(D_h)}$.\\
{\rm{(ii)}}~The authors believe that one can take the initial data, freely. It seems to hold that if $u_j^0, v_j^0 \geq 0$ with $u_{j_1}^0, v_{j_2}^0 > 0$ for some $j_1$ and $j_2$, then there exists a $k_0 \in {\mathbb N}$ such that $u_j^k, v_j^k \in (q, 1)$ for $k \geq k_0$ and $j$. This means that absorbing sets for numerical solutions exist in $S_{\it\Delta}$.\\
{\rm{(iii)}}~The similar results on the predator-prey models can be obtained. The reader can find the details on PDE in \cite{NST19} and references therein.
}\end{remark}

%%%%%%%%%%%%%%%%%%%%%%%%%%%%%%%%%%%%%%%%%%%%%%%%%%%%%%%%%%
%
% Bibliography
%
%%%%%%%%%%%%%%%%%%%%%%%%%%%%%%%%%%%%%%%%%%%%%%%%%%%%%%%%%%

\end{document}